\documentclass[pdflatex,sn-mathphys-num]{sn-jnl}

\usepackage{graphicx}%
\usepackage{multirow}%
\usepackage{amsmath,amssymb,amsfonts}%
\usepackage{amsthm}%
\usepackage{mathrsfs}%
\usepackage[title]{appendix}%
\usepackage{xcolor}%
\usepackage{textcomp}%
\usepackage{manyfoot}%
\usepackage{booktabs}%
\usepackage{algorithm}%
\usepackage{algorithmicx}%
\usepackage{algpseudocode}%
\usepackage{listings}%

\theoremstyle{thmstyleone}%
\newtheorem{theorem}{Theorem}
\newtheorem*{theorem*}{Theorem}
\newtheorem{lemma}{Lemma}

\theoremstyle{thmstyletwo}%
\newtheorem{example}{Example}%
\newtheorem{remark}{Remark}%
\newtheorem{question}{Question}

\theoremstyle{thmstylethree}%

\begin{document}

\title[Article Title]{Bergman metrics induced by the ball}

\author{\fnm{Matteo} \sur{Palmieri}} \email{matteo.palmieri@unica.it}

\affil{\orgdiv{Dipartimento di Matematica e Informatica}, \orgname{Università degli Studi di Cagliari}, \orgaddress{\street{Via Ospedale 72}, \city{Cagliari}, \postcode{09124}, \state{Sardinia}, \country{Italy}}}

\abstract{We investigate when the Bergman metric of a bounded domain is, up to a constant factor $\lambda$, induced by the Bergman metric of a finite-dimensional unit ball $\mathbb{B}^N$ via a holomorphic isometric immersion. For a strictly pseudoconvex domain in $\mathbb{C}^2$ we prove rigidity: if such an immersion extends smoothly and transversally past the boundary and 
$(N + 1)/\lambda - 3 \in \mathbb{N}$, then the domain is biholomorphic to the ball. We then consider two broad classes of examples: Hartogs domains over bounded homogeneous bases and egg domains over irreducible symmetric bases, and show that, in finite target dimension, the only members whose (rescaled) Bergman metric is induced by that of a ball are the balls themselves. The proofs combine Calabi's diastasis criterion with explicit Bergman kernel formulas (such as Fefferman's expansion) and algebraic arguments that force arithmetic constraints on the scaling factor. In higher dimensions, the first result follows under a Ramadanov-type assumption.}

\keywords{Bounded domain, Bergman kernel, Bergman metric, K{\"a}hler immersion}

\maketitle

\section{Introduction}\label{sec1}

\subsection{The Bergman kernel and metric}
Since their introduction by S. Bergman in 1922 (see \cite{Bergman1922}), the Bergman kernel and metric have covered a crucial role in several topics of geometrical analysis and differential geometry. Recall that, for a bounded domain $\Omega \subset \mathbb{C}^n$, the associated Bergman space is the Hilbert space of holomorphic square-integrable functions
\begin{align*}
    L^2_h(\Omega) = \mathcal{O}(\Omega) \cap L^2(\Omega)
\end{align*}
If $\{ \phi_\alpha \}_{\alpha \in A}$ is a complete orthonormal system for $L^2_h(\Omega)$, the \emph{Bergman kernel} of $\Omega$ is given by the expansion
\begin{equation}\label{Eqkernelexpansion}
    K_\Omega(z, \xi) = \sum_{\alpha \in A}\ \phi_\alpha(z) \overline{\phi_\alpha(\xi)}
\end{equation}
This formula is far from optimal to find closed forms of $K_\Omega$ for $\Omega$ arbitrary. Nonetheless, it is well-known (e.g. by choosing a specific complete orthonormal system) that the function $\Phi(z) := \log K_\Omega(z, z),\ z \in \Omega$, is strictly plurisubharmonic, so 
\begin{align*}
    \omega_\Omega = \sqrt{-1} \partial \bar\partial \Phi
\end{align*}
defines a K{\"a}hler form on $\Omega$. The associated K{\"a}hler metric $g_\Omega$ is called the \emph{Bergman metric} of $\Omega$. It is remarkable that if $h \colon \Omega_1 \to \Omega_2$ is a biholomorphism of bounded domains, the Bergman kernel for $\Omega_1$ transforms as
\begin{equation}\label{EqtransfBergman}
    K_{\Omega_1}(z, z) = | \det(J(h)(z)) |^2 K_{\Omega_2}(h(z), h(z))
\end{equation}
where $J(h)$ is the Jacobian matrix of $h$, which implies that biholomorphisms between bounded domains are isometries with respect to the Bergman metrics.

The most basic and notorious model is given by the unit ball $\mathbb{B}^n$, for which
\begin{equation}\label{Eqkernelball}
    K_{\mathbb{B}^n}(z, z) = \frac{n!}{\pi^n} \frac{1}{(1 - || z ||^2)^{n + 1}}\ ,\quad \omega_{\mathbb{B}^n} = (n + 1) \sqrt{-1} \partial \bar\partial \log (1 - || z ||^2)
\end{equation}
where $|| \cdot ||$ is the Euclidean norm on $\mathbb{C}^n$. Notice that, up to the scalar factor $(n + 1)$, $g_{\mathbb{B}^n}$ coincides with the hyperbolic metric on $\mathbb{B}^n$ of constant holomorphic sectional curvature $-4$. The interested reader can refer to \cite{Krantz2013} for further literature on the Bergman kernel and metric. 

Throughout the last century, mathematicians dedicated their efforts to unraveling the relations between the geometric behavior of $\Omega$ and the properties of $g_\Omega$. Great attention has been given to the general problem of understanding what conditions on $g_\Omega$ imply that $\Omega$ must be biholomorphic to $\mathbb{B}^n$.

A crucial result in this direction is \emph{Lu's Uniformization Theorem} \cite{UnifThm1}, which establishes that if $g_\Omega$ is complete with constant holomorphic sectional curvature (c.h.s.c. in the sequel) then $\Omega$ is biholomorphic to $\mathbb{B}^n$. Very recently, Huang, Li and Treuer \cite{UnifThm2} relaxed the hypothesis of completeness and assuming instead pseudoconvexity (see \cite{Bremermann1955} for the relationship between completeness and pseudoconvexity) showed that $\Omega$ is biholomorphic to $\mathbb{B}^n$ with possibly a pluripolar set removed. Shortly later, Ebenfelt, Treuer and Xiao \cite{UnifThm3} maintained only the curvature hypothesis and demonstrated that $\Omega$ is biholomorphic to $\mathbb{B}^n$ with possibly a measure zero set removed, over which $L^2_h$ functions extend holomorphically to $\mathbb{B}^n$.

Other curvature conditions gained relevance in this context, such as the Einstein condition for $g_\Omega$. In his renowned problem section \cite{YauProblem} Yau asked (in a slightly different way) if the completeness and Einstein hypotheses on $g_\Omega$ imply that $\Omega$ is homogeneous, that is, its biholomorphisms group acts transitively on $\Omega$. This is a natural question, since it follows by \cite{Bergman1951}, \cite{ChengYau} that the converse is true. Furthermore, assuming this conjecture holds, it follows by the work of Rosay \cite{Rosay} that the only smoothly bounded pseudoconvex domain with Bergman-Einstein metric is the ball. This has been cleared in the strict pseudoconvexity case, also known as the Cheng conjecture (see \cite{FuWong} for the precise statement), in dimension two by Fu and Wong \cite{FuWong} and Nemirovskii and Shafikov \cite{NemirovskiShafikovConj} and then in the general case by Huang and Xiao \cite{HuangXiao}.

In this paper, we will focus on yet another condition: the existence of a holomorphic isometric immersion (\emph{K{\"a}hler immersion} in the sequel) of $(\Omega, g_\Omega)$ into the ball. More precisely, we will address the following:
\begin{question}\label{q1}
    Let $\Omega \subset \mathbb{C}^n$ be a bounded domain. When, eventually after rescaling, is $g_\Omega$ induced by $g_{\mathbb{B}^N}$? Namely, under what conditions is there a K{\"a}hler immersion $(\Omega, \lambda g_\Omega) \to (\mathbb{B}^N, g_{\mathbb{C}^N})$?
\end{question}
This condition can be seen as a generalization of the c.h.s.c. one (a K{\"a}hler immersion $(\Omega, \lambda g_\Omega) \to (\mathbb{B}^n, g_{\mathbb{B}^n})$ is locally a biholomorphic isometry by dimensional reasons, so it preserves the holomorphic sectional curvature), and is in general much weaker. Furthermore, it has already led to promising rigidity results. Di Scala and Loi used it \cite{LoiScala1} to prove that among all the Hermitian symmetric spaces of noncompact type, the hyperbolic spaces are the only one which admits a K{\"a}hler immersion, after rescaling, into $\mathbb{B}^N$. This applies in particular to the case of $\Omega$ symmetric. Later on, Di Scala, Ishi and Loi \cite{LoiScala2} broadened the result to homogeneous K{\"a}hler manifolds (valid in particular to the case of $\Omega$ homogeneous). In \cite{HaoWang}, Hao and Wang found some constrictions on the existence of K{\"a}hler immersions into Fubini--Study spaces for bounded pseudoconvex Hartogs domains.

\subsection{Statement of the results}
The purpose of this work is to present three rigidity results for the ball involving the above mentioned condition. In Section \ref{sec2}, the first part of this article, we recall the work of Calabi \cite{Calabi53} on K{\"a}hler immersions into Fubini--Study spaces. It will be of great use to prove the main results of this article, which will be addressed in the second part of this article, namely Section \ref{sec3} and Section \ref{sec4}.

The first result we present is the following (motivational) theorem. From now on, \lq\lq\text{cl}\rq\rq\ denotes the topological closure.
\begin{theorem}\label{thm1}
Let $\Omega \subset \mathbb{C}^2$ be a simply connected strictly pseudoconvex bounded domain with $C^1$-regular boundary. If there is $\lambda > 0$ such that $\frac{N + 1}{\lambda} - 3 \in \mathbb{N}$ and a K{\"a}hler immersion $f \colon (\Omega, \lambda g_\Omega) \to (\mathbb{B}^N, g_{\mathbb{B}^N}),\ N < \infty$, that extends smoothly to $F \colon U \to V$, where $U, V$ are neighborhoods of $\text{cl}(\Omega), \text{cl}(\mathbb{B}^N)$, such that $F(\partial \Omega) \subseteq \partial \mathbb{B}^N$ transversally and $F(U \smallsetminus \text{cl}(\Omega)) \subseteq V \smallsetminus \text{cl}(\mathbb{B}^N)$, then $\Omega \cong \mathbb{B}^2$. 
\end{theorem}
The algebraic assumption on the rescaling parameter might seem unnatural. It is quite the contrary, as we will highlight in Section \ref{sec3} how the K{\"a}hler immersion condition imposes strict algebraic constraints. The proof involves Fefferman's expansion for the Bergman kernel \cite{Fefferman74}, and the key step is to show that the log-term of the expansion vanishes to infinite order when approaching $\partial \Omega$. This is achieved by combining the additional hypotheses on the K{\"a}hler immersion and the rescaling parameter. The conclusion follows by applying Ramadanov's conjecture \cite{Ramadanov}, established in dimension two by Boutet de Monvel \cite{Monvel}, which yields that $\partial \Omega$ is spherical, and the uniformization results of Nemirovskii and Shafikov \cite{NemirovskiShafikovUnif}. 

In the second part of this paper, namely Section \ref{sec3}, we demonstrate two rigidity results for families of domains that are somewhat easy to write down, but rich enough (see \cite{AzukawaSuzuki}, \cite{DAngelo1994}, \cite{HaoWang}, \cite{PanWangZhang}, \cite{IshiPark}, \cite{RoosYin}) to display nontrivial behavior of kernels, metrics, automorphisms, and boundary regularity. Hence, they function as testbeds for conjectures and counterexamples. For instance, in the counterexample of the conjecture that the universal covering of a compact K{\"a}hler manifold of negative sectional curvature should be biholomorphic to the unit ball, the explicit computation of their Bergman kernel greatly contributed (see \cite{Siu}). Great interest into them has also been shown towards the Lu Qi-Keng's problem (see \cite{Boas}, \cite{IshiPark}). Specifically, we will prove:
\begin{theorem}\label{thm2}
Up to biholomorphisms, the only Hartogs-type domain based on a bounded homogeneous domain whose Bergman metric is induced, eventually after rescaling, by the Bergman metric of $\mathbb{B}^N,\ N < \infty$, is the ball.
\end{theorem}
\begin{theorem}\label{thm3}
Up to biholomorphisms, the only egg domain based on an irreducible symmetric domain whose Bergman metric is induced, eventually after rescaling, by the Bergman metric of $\mathbb{B}^N,\ N < \infty$, is the ball.
\end{theorem}
The argument for both these results is based on the explicit form of the Bergman kernel for Hartogs-type and egg domains, obtained when the basis has a high degree of symmetry. Calabi's criterion is used to obtain algebraic relations involving the constants that define the domain, which result in strong constraints through the use of algebraic methods inspired by holomorphic Nash algebraic functions. Finally, the condition $N < \infty$ ensures rigidity for the ball.

\subsection*{Declarations}
The author has no relevant financial or non-financial interests to disclose.

\section{K{\"a}hler immersions into Fubini--Study spaces}\label{sec2}

Recall that a \emph{complex space form} is a connected complete K{\"a}hler manifold of constant holomorphic sectional curvature $c$. It is well-known (see \cite{KobayashiNomizu}, Theorem $7.9$) that there are only three models of simply connected complex space forms, up to biholomorphic isometries, distinguished by the sign of $c$. If we set $c = 4b$, these types are:
\begin{itemize}[a]
    \item the complex Euclidean space $\mathbb{C}^N,\ N \leq \infty$, with the flat metric $g_0$ ($b = 0$). Here, $\mathbb{C}^\infty := \ell^2(\mathbb{C})$ is the complex Hilbert space of sequences $z = (z_j)_{j \in \mathbb{N^+}}$ of complex numbers such that $|| z ||^2 := \sum_{k = 1}^\infty |z_k|^2 < \infty$;
    \item the complex hyperbolic space $\mathbb{CH}^N := \mathbb{B}^N_{\frac{1}{-b}},\ N \leq \infty$, with the hyperbolic metric $\frac{1}{-b}g_{hyp}$ ($b < 0$). Its K{\"a}hler form is given by
    \begin{align*}
        \omega_{hyp} = \frac{1}{b}\sqrt{-1} \partial \bar\partial \log(1 + b || z ||^2)
    \end{align*}
    Here, $\mathbb{CH}^\infty$ is the ball of radius $\frac{1}{-b}$ in $\mathbb{C}^\infty$; 
    \item the complex projective space $\mathbb{CP}^N,\ N \leq \infty$, with the Fubini--Study metric $\frac{1}{b}g_{FS}$ ($b > 0$). Its K{\"a}hler form is given in homogeneous coordinates $[Z_0, \dots, Z_N]$ in the affine charts $U_j = \{ Z_j \neq 0 \}$, $j = 0, \dots, N$ by
    \begin{align*}
        \omega_{FS} = \frac{1}{b}\sqrt{-1} \partial \bar\partial \log\left(1 + \sum_{i \neq j} \left| \frac{Z_i}{Z_j} \right|^2\right)
    \end{align*}
    Here, $\mathbb{CP}^\infty$ is the projectivization of $\mathbb{C}^\infty$.
\end{itemize}
We denote these spaces, often called \emph{Fubini--Study spaces}, by $F(N, b)$. In \cite{Calabi53}, Calabi characterized the existence and behavior of K{\"a}hler immersions of a real analytic K{\"a}hler manifold $(M, g)$ into $F(N, b)$ by means of a special function. Starting from a real analytic K{\"a}hler potential $\phi$ for $g$ and a coordinate system $\varphi$ around $p \in M$, the \emph{diastasis function} around $p$ is given by
\begin{align*}
    D(z, w) = \phi(z, \bar{z}) + \phi(w, \bar{w}) - \phi(z, \bar{w}) - \phi(w, \bar{z})
\end{align*}
where $z, w$ are coordinates in $\varphi$ of points around $p$ and $\phi(z, \bar{w})$ is the analytic continuation of $\phi$ to a neighborhood of $(p, p)$ in $M \times \overline{M}$ (here $\overline{M}$ is the conjugate of $M$). If $\varphi(p) = w_0$, the function
\begin{align*}
    D_p(z) = D(z, w_0)
\end{align*}
is still a real analytic K{\"a}hler potential for $g$ around $p$, called the \emph{diastasis potential} at $p$. It is remarkable that $D,\ D_p$ do not depend on $\phi$ nor $\varphi$, and that the isometry condition for a holomorphic map $f \colon M \to N$ of real analytic K{\"a}hler manifolds translates to the preservation of the diastasis potential (see \cite{Calabi53}, Proposition $1-6$):
\begin{equation}\label{Eqdiastasispreserved}
    f^*g^N = g^M \quad\iff\quad\forall\ p \in M:\ D_p^M = D^N_{f(p)} \circ f
\end{equation}
The interested reader can also find a recent and detailed exposition about this subject in \cite{LoiZedda}. In the present work, the following example are of great relevance:

\begin{example}\label{Ex1}
By the properties of $K_\Omega$, the K{\"a}hler manifold $(\Omega, g_\Omega)$ is real analytic. For $p \in \Omega$, the diastasis around $p$ is globally defined on $\Omega$ and given in the usual coordinates by
\begin{equation}\label{EqdiastasisBergman}
    D^\Omega(z, w) = \log\left( \frac{K_\Omega(z, z) K_\Omega(w, w)}{|K_\Omega(z, w)|^2}\right)
\end{equation}
\end{example}
\begin{example}\label{Ex2}
For $p \in F(N, b),\ b \neq 0$: the diastasis around $p$ is given in Bochner coordinates centered at $p$ by
\begin{equation}\label{Eqdiastasishyperbolic}
    D^b(z, w) = \frac{1}{b} \log(1 + b \sum_{i = 1}^N z_i \bar{w}_i)
\end{equation}
Moreover, if $b > 0$: the diastasis around any point can be written in homogeneous coordinates $Z = [Z_0, \dots, Z_N],\ W = [W_0, \dots, W_N]$ as
\begin{equation}\label{Eqdiastasisprojective}
    D^b(Z, W) = \frac{1}{b} \log \left( \frac{(\sum_{i = 0}^N | Z_i |^2 )(\sum_{i = 0}^N | W_i |^2 )}{(\sum_{i = 0}^N Z_i \overline{W}_i )} \right)
\end{equation}
\end{example}
\begin{example}\label{Ex3}
If $(M, g)$ is a real analytic K{\"a}hler manifold, then $\forall\ \lambda > 0$: $(M, \lambda g)$ is again real analytic and if $D^g,\ D^{\lambda g}$ are the diastasis functions for $g, \lambda g$
\begin{equation}\label{Eqdiastasisrescaling}
    D^{\lambda g} = \lambda D^g
\end{equation}
\end{example}

The main tool we need from Calabi's work is the so-called \emph{Calabi's local criterion} (see \cite{Calabi53}, Theorem $8$), which we state in the context of our interest. Suppose $\text{dim}(M) = n$, and fix the lexicographic order for the set of multi-indexes $\mathbb{N}^n$. For $b \neq 0$, $p \in M$ and coordinates $z = z_1, \dots, z_n$ around $p$, consider the expansion on a neighborhood of $p$
\begin{align*}
    \frac{e^{bD_p(z)} - 1}{b} = \sum_{j, k = 0}^\infty\ s_{jk}(z(p))\ (z - z(p))^{m_j} \overline{(z - z(p))^{m_k}}
\end{align*}
where we adopted the multi-index convention $w^{m_l} = w_1^{m_{l1}} \cdot \ldots \cdot w_n^{m_{ln}}$. Then, it holds:
\begin{theorem}[\textbf{Calabi's local criterion}]\label{thmCalabicriterion}
There is a neighborhood of $p$ which admits, when endowed with the induced metric of $M$, a K{\"a}hler immersion into $F(N, b),\ b < 0$, if and only if the infinite-dimensional matrix $(s_{jk}(z(p)))_{j, k}$ is positive semidefinite of rank at most $N$.
\end{theorem}
Since the diastasis depends only on $g$, the statement of the theorem is independent of the choice of coordinates around $p$. If $M$ is connected, it is proved (see \cite{Calabi53}, Theorem $10$) that the minimum possible dimension $N$ for which any point of $M$ has a neighborhood which can be K{\"a}hler immersed into $F(N, b)$ also depends only on $g$. We refer to these (local) K{\"a}hler immersions as \emph{full}, and any two full (local) K{\"a}hler immersions into the same Fubini--Study space differ by a rigid motion of the codomain: this fact is known as \emph{Calabi's rigidity} (see \cite{Calabi53}, Theorem $9$).

\begin{remark}\label{rmk1}
The spaces $F(N, b)$ act as ambient spaces in which real analytic K{\"a}hler manifolds live. We remark here that here is always a K{\"a}hler immersion $(\Omega, g_\Omega) \to F(\infty, 1)$, for if $\{ \phi_\alpha \}_{\alpha \in N^+}$ is a complete orthonormal system for $L^2_h(\Omega)$, the full holomorphic map
\begin{align*}
    f \colon \Omega \to F(\infty, 1) : z \mapsto [\phi_1(z) : \phi_2(z) : \dots ]
\end{align*}
satisfies $D^1_{f(p)} \circ f = D^\Omega_{p}\,\ \forall\ p \in \Omega$, as it can be easily seen due to \eqref{Eqkernelexpansion}, \eqref{EqdiastasisBergman}, \eqref{Eqdiastasisprojective}. It follows by Calabi's rigidity that $(\Omega, g_\Omega)$ cannot be K{\"a}hler immersed, not even after rescaling, into $F(N, b)$ with both $b > 0$ and $N < \infty$. On the other hand, we are well-aware of the existence of a holomorphic immersion of $(\Omega, g_\Omega)$ into $F(N, b)$ for $N \geq n,\ b < 0$, since $\Omega$ is bounded. Hence, Question \ref{q1} emerges as very natural from this viewpoint.
\end{remark}

\section{Transversal extensions of K{\"a}hler immersions}\label{sec3}

\subsection{Fefferman's expansion}

Let $\Omega \subset \mathbb{C}^n$ be a smoothly bounded and strictly pseudoconvex domain. By a profound result of Fefferman (see \cite{Fefferman74}, Corollary at p. $45$) there is an asymptotic expansion
\begin{equation}\label{EqFefferman}
    K_\Omega(z, z) = \frac{\phi_\Omega(z)}{\psi_\Omega^{n + 1}(z)} + \widetilde{\phi_\Omega}(z) \log (\psi_\Omega(z))\ ,\quad z \in \Omega
\end{equation}
where $\phi_\Omega,\ \widetilde{\phi_\Omega} \in C^\infty(\text{cl}(\Omega))$ with $\phi_\Omega |_{\partial \Omega} \neq 0$ and $\psi \in C^\infty(\text{cl}(\Omega))$ is a \emph{defining function} for $\Omega$, that is, it satisfies
\begin{align*}
    \Omega = \{ \psi > 0 \}\ ,\quad \partial \Omega = \{ \psi = 0\}\ ,\quad \nabla \psi |_{\partial \Omega} \neq 0
\end{align*}
The asymptotic behavior of $K_\Omega$ near any point of $\partial \Omega$ depends only on the local CR geometry of $\partial \Omega$ at the point, as it was highlighted in \cite{Fefferman74}. This fact is reflected by the properties of $\phi_\Omega,\ \widetilde{\phi_\Omega}$: we will now briefly recall one relevant instance for our discussion.

Comparing \eqref{Eqkernelball} and \eqref{EqFefferman}, we see that $\widetilde{\phi}_{\mathbb{B}^n}$ is identically zero. What's more, it can be proved that the coefficient $\widetilde{\phi_\Omega}$ of the logarithmic term in \eqref{EqFefferman} vanishes to infinite order at the boundary for $\partial \Omega$ spherical. Ramadanov conjectured \cite{Ramadanov} that the converse is also true, namely the vanishing to infinite order at the boundary of $\widetilde{\phi_\Omega}$ implies that $\partial \Omega$ is spherical. This result was cleared by Boutet de Monvel \cite{Monvel} in dimension two, while the general case is still widely open. It is also worth mentioning, within the scope of this article, that the Cheng conjecture actually follows from the Ramadanov conjecture (see \cite{NemirovskiShafikovConj}, Theorem at p. $781$).  

\subsection{Transversality}

Let $\Omega_1 \subset \mathbb{C}^n, \Omega_2 \subset \mathbb{C}^m$ be bounded domains with $C^1$-regular boundaries, $A_1 \subset \mathbb{C}^n$, $A_2 \subset \mathbb{C}^m$ be neighborhoods of $\text{cl}(\Omega_1), \text{cl}(\Omega_2)$ respectively. A map $\Xi \in C^1(A_1, A_2)$ with $\Xi(\partial \Omega_1) \subset \partial \Omega_2$ is \emph{transverse} to $\partial \Omega_2$ at $x \in \partial \Omega_1$ if
\begin{align*}
    (D_x\Xi)(T_x\Omega_1) + T_{\Xi(x)} \partial \Omega_2 = T_{\Xi(x)} \Omega_2
\end{align*}
$\Xi$ is \emph{transverse along} $\partial \Omega_2$ if it is transverse to $\partial \Omega_2$ at any point of $\partial \Omega_1$. Equivalently, let $\rho_i \in C^1(A_i)$ be a defining function for $\Omega_i,\ i = 1, 2$. Then
\begin{align*}
    \nabla \rho_i\ \text{is normal to } \partial \Omega_i,\ i = 1, 2
\end{align*}
and $\Xi$ is transverse to $\partial \Omega_2$ at $x \in \partial \Omega_1$ if and only if
\begin{align*}
    \langle (D_x\Xi)(\nabla \rho_1(x)), \nabla \rho_2(\Xi(x)) \rangle \neq 0
\end{align*}
In order to implement transversality, notice that the following identity holds:
\begin{equation}\label{Eqtransv}
    \begin{split}
        \langle (D_x\Xi)(\nabla \rho_1(x)), \nabla \rho_2(\Xi(x)) \rangle &= \langle \nabla \rho_1(x), (D_x\Xi)^T \nabla \rho_2(\Xi(x)) \rangle =\\
        &= \langle \nabla \rho_1(x), \nabla(\rho_2 \circ \Xi)(x) \rangle
    \end{split}
\end{equation}

\subsection{Proof of Theorem \ref{thm1}}

The first part of the proof can be carried out in any dimension $n$. Up applying a translation of $\mathbb{C}^n$ and a rigid motion of $\mathbb{B}^N$, we can assume $0 \in \Omega$ and $f(0) = 0$. Since $(\mathbb{B}^N, g_{\mathbb{B}^N}) \equiv F(N, -\frac{1}{N + 1})$, up to homotheties we can pick as defining function for $F(N, -\frac{1}{N + 1})$ 
\begin{align*}
    \psi_{\mathbb{B}^N}(z) = 1 - \frac{1}{N + 1} || z ||^2\ ,\quad z \in F(N, -\frac{1}{N + 1})
\end{align*}
By \eqref{EqdiastasisBergman}, \eqref{Eqdiastasishyperbolic}, \eqref{Eqdiastasisrescaling}, the diastasis preservation \eqref{Eqdiastasispreserved} through $f$ at $0$ can be equivalently written as
\begin{equation}\label{EqdiastasisFefferman}
    K_\Omega(z, z) = \frac{| K_\Omega(z, 0) |^2}{K_\Omega(0, 0)} \frac{1}{\psi_{\mathbb{B}^N}(f(z))^{\frac{N + 1}{\lambda}}}\ ,\quad z \in \Omega
\end{equation}
Since $F(\Omega) = f(\Omega) \subset \mathbb{B}^N,\ F(\partial \Omega) \subset \partial \mathbb{B}^N,\ F(U \smallsetminus \text{cl}(\Omega)) \subseteq V \smallsetminus \text{cl}(\mathbb{B}^N)$, the function $\psi_\Omega := \psi_{\mathbb{B}^N} \circ F \in C^\infty(\text{cl}(\Omega))$ satisfies
\begin{align*}
    \Omega = \{\psi_\Omega > 0\}\ ,\quad \partial \Omega = \{ \psi_\Omega = 0 \}
\end{align*}
Furthermore, from \eqref{Eqtransv}, we deduce that transversality of $F$ yields $\nabla \psi_\Omega |_{\partial \Omega} \neq 0$. Hence, $\psi_\Omega$ is a defining function for $\Omega$ (so $\Omega$ is smoothly bounded), and comparing \eqref{EqFefferman} with \eqref{EqdiastasisFefferman} we get
\begin{align*}
    \frac{| K_\Omega(z, 0) |^2}{K_\Omega(0, 0)} - \phi_\Omega(z) (\psi_\Omega(z))^{\frac{N + 1}{\lambda} - n - 1} - (\psi_\Omega(z))^{\frac{N + 1}{\lambda}} \widetilde{\phi_\Omega}(z) \log (\psi_\Omega(z)) = 0\ ,\quad z \in \Omega
\end{align*}
If $\frac{N + 1}{\lambda} - n - 1 \in \mathbb{N}$, by Lemma $2.2$ of \cite{FuWong} we infer that $\psi_\Omega^{\frac{N + 1}{\lambda}} \widetilde{\phi_\Omega}$ vanishes to infinite order at $\partial \Omega$. Being $\nabla \psi_\Omega |_{\partial \Omega} \neq 0$, we conclude that
\begin{align*}
    \widetilde{\phi_\Omega}\ \text{vanishes to infinite order at}\ \partial \Omega
\end{align*}
In the setting $n = 2$, this fact implies that $\partial \Omega$ is spherical, by Ramadanov's conjecture. By Theorem $A.2$ in \cite{NemirovskiShafikovUnif}, $\Omega$ is universally covered by $\mathbb{B}^n$, which yields $\Omega \cong \mathbb{B}^n$ being $\Omega$ simply connected. 

\begin{remark}\label{rmk2}
Assuming Ramadanov's conjecture is true, the above argument (thus Theorem \ref{thm1}) holds in any dimension.
\end{remark}

\section{Rigidity of K{\"a}hler immersions into the ball }\label{sec4}

\subsection{Hartogs-type and egg domains}

In Theorem \ref{thm2} and Theorem \ref{thm3}, we focus on the following domains built upon sufficiently regular bounded domains. We will adopt the single-argument notation for the Bergman kernel on the diagonal to ease the notation. The complete construction can be found, for each case, in \cite{IshiPark}, \cite{RoosYin} respectively. 

For the second result of this work, let $\Omega \subset \mathbb{C}^n$ be a bounded homogeneous domain. We can associate some holomorphic invariants to $\Omega$, which arise by transferring the Hermitian structure given pointwisely by $g_\Omega$ on $T\Omega$ to the Lie algebra of a maximal connected split solvable subgroup of the holomorphic automorphism group of $\Omega$:
\begin{align*}
    r \in \mathbb{N}^*\ ,\quad \{ p_k, q_k, b_k \}_{k = 1, \dots, r} \subset \mathbb{N}
\end{align*}
The constant $r$ is called \emph{rank} of $\Omega$. If, for any $k = 1, \dots, r$, we set
\begin{align*}
    a_{ki} := \frac{i + \frac{q_k}{2}}{p_k + q_k + b_k + 2}\ ,\quad 1 \leq i \leq 1 + p_k + b_k
\end{align*}
we can consider $C_\Omega := \underset{i, k}{\min}\{ a_{ki}\} > 0$ and the polynomials
\begin{equation}\label{EqpolynomialHartogs}
    F(s) := \prod_{k, i} \left(1 + \frac{s}{a_{ki}}\right)\ ,\quad b(k) = F(sk)
\end{equation}
of degree $d$. For $m \in \mathbb{N^*}$ and $s \in \mathbb{R}$, define the \emph{Hartogs-type domain} based on $\Omega$ as
\begin{align*}
    \Omega_{m, s} := \{ (z, \xi) \in \Omega \times \mathbb{C}^m : || \xi ||^2 < K_\Omega(z, z)^{-s} \}
\end{align*}
where $s > - C_\Omega$ is a real parameter. If in the rising factorial basis $\{ (k + 1)_j \}_{j = 0, \dots, d}$, where $(x)_l = x(x + 1) \cdot \ldots \cdot (x + l - 1)$ is the Pochhammer symbol, we have
\begin{align*}
    b(k) = \sum_{i = 0}^d\ c(s, j) (k + 1)_j
\end{align*}
then the Bergman kernel for $\Omega_{m, s}$ is given on the diagonal of $\Omega_{m, s} \times \Omega_{m, s}$ by (see \cite{IshiPark}, Theorem $4.4$)
\begin{equation}\label{EqBergmanHartogs}
    K_{\Omega_{m, s}}\left(z, \xi\right) = \frac{K_\Omega(z)^{ms + 1}}{\pi^m} \sum_{j = 0}^d \frac{c(s, j)(j + m)!}{(1 - t)^{j + m + 1}} \Big|_{t = K_\Omega(z)^s || \xi ||^2} 
\end{equation}

For the third result, let $\Omega \subset \mathbb{C}^n$ be a bounded irreducible symmetric domain. We can associate some holomorphic invariants to $\Omega$, which arise from the Hermitian Positive Jordan Triple System on $\mathbb{C}^n$ induced by $\Omega$:
\begin{align*}
    r \in \mathbb{N}^*\ ,\quad a, b \in \mathbb{N}
\end{align*}
The constant $r$ is called \emph{rank} of $\Omega$. Define the \emph{genus} of $\Omega$ by $g := 2 + a(r - 1) + b \in \mathbb{N}^*$ and the \emph{generic norm} of $\Omega$ by
\begin{align*}
    N_{\Omega}(z) = \left( \text{vol}(\Omega) K_\Omega(z) \right)^{-\frac{1}{g}}\ ,\quad z \in \Omega
\end{align*}
For $p, q \in \mathbb{N}^*$ and $k \in \mathbb{R},\ k > 0$, define the \emph{egg domains} over $\Omega$ of type $I, II$ respectively
\begin{enumerate}
    \item[I)] $Y(q, \Omega, k) := \{ (z, \xi) \in \Omega \times \mathbb{C}^q : || \xi ||^{2k} < N_\Omega(z) \}$
    \item[II)] $E(p, q, \Omega, k) := \{ (z, \xi_1, \xi_2) \in \Omega \times \mathbb{C}^p \times \mathbb{C}^q : || \xi_1 ||^2 + || \xi_2 ||^{2k} < N_\Omega(z) \}$
\end{enumerate}
\begin{remark}\label{rmk3}
Since any bounded symmetric domain $\Omega$ is homogeneous and $Y(q, \Omega, k) \cong \Omega_{q, \frac{1}{kg}}$, we only need to focus on egg domains of type $II$. We will simply address them as egg domains.
\end{remark}
Consider the polynomial given by
\begin{align*}
    \chi(s) = \prod_{j = 1}^r \left( s + 1 + (j - 1)\frac{a}{2} \right)_{1 + b + (r - j)a}
\end{align*}
If in the rising factorial basis $\{(h + 1)_j\}_{j = 1, \dots, n + 2}$ we have
\begin{align*}
    h(h - 1) \chi(h) = \sum_{j = 1}^{n + 2}\ b_j\ (h + 1)_j
\end{align*}
we can well-define the functions
\begin{align*}
    H_{jm}(\lambda) &:= \sum_{l = 0}^\infty \left( \frac{l + 1}{k} + 2 + m \right)_{j - m} \lambda^l\ ,\\
    H(t_1, \lambda) &:= \sum_{j = 1}^{n + 2} b_j (1 - t_1)^{-j} \sum_{m = 0}^j (-j)_m (2)_m \frac{t_1^m}{\Gamma(m + 1)} H_{jm}(\lambda)
\end{align*}
and consider the function
\begin{align*}
    \Lambda(t_1, t_2) := \frac{1}{\text{vol}(\Omega)} \frac{k}{\chi(0)} (1 - t_1)^{-\frac{1}{k}} H\left( t_1, \frac{t_2}{(1 - t_1)^{\frac{1}{k}}} \right)
\end{align*}
Denoting the partial derivatives of $\Lambda$ by $\Lambda^{(p - 1), (q - 1)}(t_1, t_2) = \frac{\partial^{p + q - 2} \Lambda}{(\partial t_1)^{p - 1}(\partial t_2)^{q - 1}}(t_1, t_2)$, the Bergman kernel for $E(p, q, \Omega, k)$ is given on the diagonal of $E(p, q, \Omega, k) \times E(p, q, \Omega, k)$ by (see \cite{RoosYin}, Corollary $3.8$)
\begin{equation}\label{EqBergmanEgg}
    K_E(z, \xi_1, \xi_2) = \frac{1}{p! q!} \Lambda^{(p - 1), (q - 1)} \left( \frac{|| \xi_1 ||^2}{N_\Omega(z)}, \frac{|| \xi_2 ||^2}{N_\Omega(z)^{\frac{1}{k}}} \right) N_\Omega(z)^{- p - \frac{q}{k} - g}
\end{equation}
\begin{remark}\label{rmk4}
The closed form \eqref{EqBergmanEgg} is considered with respect to a rescaling of the volume form, which gives volume $1$ to unit balls (see \cite{RoosYin}). This choice does not affect the results, as the Bergman kernel will only differ by a multiplicative constant, leaving metric and diastasis unchanged. Hence, we will tacitly omit this detail later on.
\end{remark}
\begin{remark}\label{rmk5}
As anticipated, Hartogs-type and egg domains are generalization of $\mathbb{B}^N$, which is (biholomorphically) retrieved as
\begin{itemize}
    \item Hartogs-type domain over $\mathbb{B}^n$, with $m = N - n$ and $s = \frac{1}{n +1}$;
    \item egg domain over $\mathbb{B}^n$, with $p + q = N - n$ and $k = 1$
\end{itemize}
The constants introduced for the second result are, in the case of $\mathbb{B}^n$, given by (see \cite{IshiPark}, Example $3.1$)
\begin{equation}\label{Eqconstball1}
    r = 1,\ p_1 = q_1 = 0,\ b_1 = n - 1,\ a_{11} = \frac{1}{n + 1} = C_{\mathbb{B}^n},\ a_{1i} = \frac{i}{n + 1}\quad \forall\ i = 2, \dots, n
\end{equation}
while the constants introduced for the third result are, in the case of $\mathbb{B}^n$, given by (see \cite{RoosYin}, Example $4.1$)
\begin{equation}\label{Eqconstball2}
    r = 1,\ a = 2,\ b = n - 1,\ g = n + 1
\end{equation}
\end{remark}
With the newly introduced notation, we can now state Theorem \ref{thm2} and Theorem \ref{thm3} more precisely:
\begin{theorem*}[\ref{thm2}, \textbf{precise form}]\label{thm2precise}
Let $\Omega \subset \mathbb{C}^n$ be a bounded homogeneous domain. If, for some $\lambda > 0$, there is a K{\"a}hler immersion $(\Omega_{m, s}, \lambda g_{\Omega_{m, s}}) \to (\mathbb{B}^N, g_{\mathbb{B}^N})$, $n + m \leq N < \infty$, then $\Omega_{m, s} \cong \mathbb{B}^{n + m}$.
\end{theorem*}
\begin{theorem*}[\ref{thm3}, \textbf{precise form}]\label{thm3precise}
Let $\Omega \subset \mathbb{C}^n$ be a bounded irreducible symmetric domain. If, for some $\lambda > 0$, there is a K{\"a}hler immersion $(E(p, q, \Omega, k), \lambda g_{E}) \to (\mathbb{B}^N, g_{\mathbb{B}^N})$, $n + p + q \leq N < \infty$, then $E(p, q, \Omega, k) \cong \mathbb{B}^{n + p + q}$.
\end{theorem*}

\subsection{Holomorphic Nash algebraic functions}

Let $V$ be a complex vector space of dimension $t \geq 1$, $A \subset V$ be open, $f$ holomorphic on $A$. Recall that $f$ is called \emph{Nash algebraic} at $x \in A$ if there is a neighborhood $U$ of $x$ and a polynomial $P \colon V \times \mathbb{C} \to \mathbb{C},\ P \neq 0$, such that 
\begin{align*}
    P(y, f(y)) = 0\quad \forall\ y \in U
\end{align*}
and $f$ is called a \emph{holomorphic Nash algebraic function} if it is Nash algebraic at each point of $A$. The interested reader can refer to \cite{NashAlgebr} for further literature on the topic.

Let $\mathcal{N}^t$ be the set of real analytic functions defined in a neighborhood $U \subset \mathbb{C}^t$ of $0 \in \mathbb{C}^t$ whose real analytic continuation to a neighborhood of $(0, 0) \in \mathbb{C}^t \times \overline{\mathbb{C}^t}$ is a holomorphic Nash algebraic function. Set
\begin{align*}
    \mathcal{F}^t = \{ \eta(f_1, \dots, f_t) : \eta \in \mathcal{N}^t,\ f_j \in \mathcal{O}_0 \text{ and } f_j(0) = 0\ \forall\ j = 1, \dots, t \}
\end{align*}
where $\mathcal{O}_0$ denotes the germ of holomorphic functions at $0 \in \mathbb{C}$, consider $\mathcal{G}^t$ the set of real analytic functions around $0$ whose expansion at any point does not contain non constant purely holomorphic or anti-holomorphic terms, and define
\begin{align*}
    \mathcal{H}^t := \mathcal{F}^t \cap \mathcal{G}^t 
\end{align*}
Then, the following theorem holds (see \cite{LoiMossa}, Theorem $2.2$):
\begin{theorem}\label{thmNash}
Let $\psi_l = \eta_l(f_{l, 1}, \dots, f_{l, t}) \in \mathcal{H}^t,\ l = 0, \dots, s$, satisfy for some $c_1, \dots, c_s \in \mathbb{R}$
\begin{align*}
    \psi_1^{c_1} \cdot \ldots \cdot \psi_s^{c_s} = \psi_0\ ,\quad \psi_0(0) \neq 0
\end{align*}
If $\{ c_1, \dots, c_s, 1 \}$ are linearly independent over $\mathbb{Q}$, then $\psi_l$ is constant $\forall\ l = 1, \dots, s$.
\end{theorem}
For our purposes, we will need an immediate consequence of this result, which we state below for clarity:
\begin{lemma}\label{LemmaNash}
Let $\psi_l = \eta_l(f_{l, 1}, \dots, f_{l, t}) \in \mathcal{H}^t,\ l = 0, 1$, satisfy for some $c \in \mathbb{R}$
\begin{align*}
    \psi_1^{c} = \psi_0\ ,\quad \psi_0(0) \neq 0
\end{align*}
If $\psi_1$ is not constant, then $c \in \mathbb{Q}$.
\end{lemma}

\subsection{Proof of Theorem \ref{thm2}}

Let $f \colon (\Omega_{m, s}, \lambda g_{\Omega_{m, s}}) \to (\mathbb{B}^N, g_{\mathbb{B}^N}),\ f = (f_1, \dots, f_N)$, be a K{\"a}hler immersion. Up to applying a translation of $\mathbb{C}^n$ and a rigid motion of $\mathbb{B}^N$, we can assume $0 \in \Omega$ (so $0 \in \Omega_{m, s}$) and $f(0) = 0$. In view of Remark \ref{rmk5}, our goal is to show
\begin{align*}
    i)\ \Omega_{m, s} \cong (\mathbb{B}^n)_{m, s}\ ,\quad ii)\ s = \frac{1}{n + 1}
\end{align*}
We immediately notice that $s \neq 0$. Indeed, since $\Omega_{m, 0} = \Omega \times \mathbb{B}^m$ and the Bergman metric respects the product structure 
\begin{align*}
    (\Omega_{m, 0}, \lambda g_{\Omega_{m, 0}})\ \text{does not admit a K{\"a}hler immersion into } \mathbb{B}^N
\end{align*}
(see the proof of Theorem $2$ in \cite{LoiScala2}). Next, the inclusion $\iota \colon \Omega \to \Omega_{m, s}: \iota(z) = (z, 0)$ is a holomorphic immersion and satisfies by \eqref{EqBergmanHartogs}
\begin{align*}
    (\iota^*K_{\Omega_{m, s}})(z) = \frac{m!}{\pi^m}\ b(m)\ K_\Omega(z)^{ms + 1}
\end{align*}
which implies $\iota^*\omega_{\Omega_{m, s}} = (ms + 1) \omega_\Omega$, where $\omega_{\Omega_{m, s}}, \omega_\Omega$ are the K{\"a}hler forms associated to the Bergman metrics. Hence, $\iota \colon (\Omega, (ms + 1) g_\Omega) \to (\Omega_{m, s}, g_{\Omega_{m, s}})$ is a K{\"a}hler immersion. It follows that
\begin{align*}
    f \circ \iota \colon (\Omega, \lambda (ms + 1) g_\Omega) \to (\mathbb{B}^N, g_{\mathbb{B}^N})\ \text{is a K{\"a}hler immersion}
\end{align*}
which implies that $\Omega \cong \mathbb{B}^n$, by Theorem $2$ of \cite{LoiScala2} (applied after rescaling). We will now prove $i)$. To do so, let $\varphi \colon \Omega \to \mathbb{B}^n$ be a biholomorphism. Applying \eqref{EqtransfBergman}, we can write
\begin{equation}\label{HartogsOmegatoball1}
    \Omega_{m, s} = \{ (z, \xi) \in \Omega \times \mathbb{C}^m : |\det(J(\varphi)(z))|^{2s} || \xi ||^2 < K_{\mathbb{B}^n}(\varphi(z))^{- s} \}
\end{equation}
We will make use of the following:
\begin{lemma}\label{Lemmalogarithm}
Let $D \subset \mathbb{C}^n$ be a simply connected domain. If $h$ is holomorphic and nowhere zero on $D$, there is a holomorphic function $\widetilde{h}$ on $D$ such that $e^{\widetilde{h}} = h$. 
\end{lemma}
\begin{proof}
Consider the holomorphic $1$-form on $D$
\begin{align*}
    \eta := \frac{1}{h} \partial h = \frac{1}{h} dh
\end{align*}
which is closed, as $d \eta = -\frac{1}{h^2} dh \wedge dh = 0$. Being $D$ simply connected, by Poincaré's Lemma there is $h' \in C^\infty(D, \mathbb{C})$ such that $d h' = \eta$. In particular, $h'$ is holomorphic because $\eta$ has bidegree $(1, 0)$. Then
\begin{align*}
    d (e^{-h'} h) = (- e^{-h'} h) d h' + e^{h'} d h = 0
\end{align*}
so there is $c \in \mathbb{C}^*$ such that $e^{h'} = c h$. Choosing a logarithm for $\frac{1}{c}$ yields the claim with $\widetilde{h} = \log(\frac{1}{c}) h'$.
\end{proof}
Being $\varphi$ a biholomorphism, $\det(J(\varphi))$ is holomorphic and nowhere zero on $\Omega$, so by Lemma \ref{Lemmalogarithm} there is a holomorphic function $\Phi$ on $\Omega$ such that $e^\Phi = \det(J(\varphi))$. Then, for $s \in \mathbb{R}$
\begin{align*}
    \det(J(\varphi))^s = e^{s \Phi}\ \text{is holomorphic on } \Omega
\end{align*}
and since \eqref{HartogsOmegatoball1} becomes
\begin{align*}
    \Omega_{m, s} = \{ (z, \xi) \in \Omega \times \mathbb{C}^m : || \det(J(\varphi)(z))^s \xi ||^2 < K_{\mathbb{B}^n}(\varphi(z))^{- s} \}
\end{align*}
we have a well-defined holomorphic map
\begin{align*}
    \Psi \colon \Omega_{m, s} \to (\mathbb{B}^n)_{m, s} : (z, \xi) \mapsto \Big( \varphi(z), \det(J(\varphi)(z))^s \xi \Big)
\end{align*}
In particular, $\Psi$ is actually a biholomorphism whose inverse is given by
\begin{align*}
    \Theta \colon (\mathbb{B}^n)_{m, s} \to \Omega_{m, s} : (Z, W) \mapsto \Big(\varphi^{-1}(Z), \det(J(\varphi)(\varphi^{-1}(Z)))^{-s} W\Big)
\end{align*}
meaning that $i)$ holds. 

It remains to prove $ii)$, which will be the most technical part of the proof. We will find a constriction on $\lambda$, and then use it together with Theorem \ref{thmCalabicriterion} to ensure the desired value for $s$. 

Since we reduced to $(\mathbb{B}^n)_{m, s}$, from \eqref{Eqconstball1} we infer the following closed form for the polynomial $b$ in \eqref{EqpolynomialHartogs}:
\begin{equation}\label{EqpolynomialHartogsball}
    b(k) = \prod_{i = 1}^n \left( 1 + \frac{(n + 1)sk}{i} \right) = \frac{1}{n!} \prod_{i = 1}^n (i + (n + 1)sk)
\end{equation}
In particular, $\text{deg}(b(k)) = n$ because $s \neq 0$, and since $(k + 1)_n$ is the only $n^{th}$-degree polynomial in the rising factorial basis, we conclude
\begin{equation}\label{Eqtopconstant}
    c(s, n) \neq 0
\end{equation}
Now, let $D_0$ be the diastasis potential of $g_{\mathbb{B}^n_{m, s}}$ at $(0, 0)$. By \eqref{EqBergmanHartogs}, \eqref{EqdiastasisBergman}
\begin{equation}\label{EqdiastasisHartogs}
    \begin{split}
        D_0(z, \xi) &= (ms + 1) \log \left( K_{\mathbb{B}^n}(z) K_{\mathbb{B}^n}(0) K_{\mathbb{B}^n}(z, 0)^{-1} K_{\mathbb{B}^n}(0, z)^{-1} \right) +\\
        &\hspace{5mm} + \log \left( \sum_{j = 0}^n c(s, j) (j + m)! \left(1 - K_{\mathbb{B}^n}(z)^s || \xi ||^2 \right)^{- j - m - 1} \right) +\\
        &\hspace{5mm} - \log \left( \sum_{j = 0}^n c(s, j) (j + m)! \right)
    \end{split}
\end{equation}
Set $S := \sum_{j = 0}^n c(s, j) (j + m)! = \frac{\pi^m K_{{\mathbb{B}^n}_{m, s}}(0, 0)}{K_{\mathbb{B}^n}(0, 0)^{ms + 1}} > 0$ and $\forall\ \gamma = 0, \dots, n : c'(s, \gamma) := \frac{c(s, \gamma)}{S}$. Then $\sum_{j = 0}^n c'(s, j) (j + m)! = 1$ and \eqref{EqdiastasisHartogs} can be written as 
\begin{equation}\label{EqdiastasisHartogsnormalized}
    \begin{split}
        D_0(z, \xi) &= (ms + 1) \log \left( K_{\mathbb{B}^n}(z) K_{\mathbb{B}^n}(0) K_{\mathbb{B}^n}(z, 0)^{-1} K_{\mathbb{B}^n}(0, z)^{-1} \right) +\\
        &\hspace{5mm} + \log \left( \sum_{j = 0}^n c'(s, j) (j + m)! \left(1 - K_{\mathbb{B}^n}(z)^s || \xi ||^2 \right)^{- j - m - 1} \right)
    \end{split}
\end{equation}
Observe that $(\mathbb{B}^N, g_{\mathbb{B}^N}) \equiv F(N, -\frac{1}{N + 1})$. By \eqref{EqdiastasisHartogsnormalized}, \eqref{Eqdiastasishyperbolic} and \eqref{Eqdiastasisrescaling}, the preservation of the diastasis \eqref{Eqdiastasispreserved} through $f$ at $(0, 0)$ reads
\begin{equation}\label{Eqdiastasis1}
    \begin{split}
        (K_{\mathbb{B}^n}&(z) K_{\mathbb{B}^n}(0) K_{\mathbb{B}^n}(z, 0)^{-1} K_{\mathbb{B}^n}(0, z)^{-1})^{-\frac{\lambda(ms + 1)}{N + 1}} \times\\
        &\hspace{5mm} \times \left( \sum_{j = 0}^n c'(s, j) (j + m)! \left(1 - K_{\mathbb{B}^n}(z)^s || \xi ||^2 \right)^{- j - m - 1} \right)^{-\frac{\lambda}{N + 1}} =\\
        &= 1 - \frac{1}{N + 1} \sum_{j = 1}^N | f_j(z, \xi) |^2
    \end{split}
\end{equation}
Set $z_1 := (1, 0, \dots, 0) \in \mathbb{C}^m$ and $C := (K_{\mathbb{B}^n}(0))^s > 0$. The complex segment $\underline{s} = \{ (0, tz_1) \in \mathbb{C}^{n + m} : t \in B_{\frac{1}{C}}(0) \subset \mathbb{C} \}$, is contained in $(\mathbb{B}^n)_{m, s}$, so we can restrict to $\underline{s}$ in \eqref{Eqdiastasis1} to get
\begin{equation}\label{EqHartogsNashalg}
    1 - \frac{1}{N + 1} \sum_{j = 1}^N | f_j(0, t z_1) |^2 = \left( \sum_{j = 0}^n c'(s, j) (j + m)! \left(1 - C | t |^2 \right)^{- j - m - 1} \right)^{-\frac{\lambda}{N + 1}}
\end{equation}
Consider $\eta \in \mathcal{G}^N$ given by
\begin{align*}
    \eta(w) = 1 - \frac{1}{N + 1} || w ||^2 = 1 - \frac{1}{N + 1} \sum_{l = 1}^N w_l \overline{w}_l
\end{align*}
whose analytic continuation around $(0, 0)$, given by
\begin{align*}
    \widetilde{\eta}(w, \overline{v}) = 1 - \frac{1}{N + 1} \langle w, v \rangle = 1 - \frac{1}{N + 1} \sum_{l = 1}^N w_l \overline{v}_l
\end{align*}
is holomorphic on $\mathbb{C}^N \times \overline{\mathbb{C}^N}$. The polynomial $P$ on $(\mathbb{C}^N \times \overline{\mathbb{C}^N}) \times \mathbb{C}$ given by 
\begin{align*}
    P((X, \overline{Y}), \alpha) = \frac{1}{N + 1} \sum_{l = 1}^N X_l \overline{Y}_l + \alpha - 1
\end{align*}
satisfies $P((w, \overline{v}), \widetilde{\eta}(w, \overline{v})) = 0$, so $\eta \in \mathcal{N}^N$. Moreover, since for $t \in B_{\frac{1}{C}}(0) \subset \mathbb{C}$: $f_1(0, t z_1), \dots, f_N(0, t z_1) \in \mathcal{O}_0$ and $f_1(0, 0) = \ldots = f_N(0, 0) = 0$, we conclude
\begin{align*}
    \psi_0 := 1 - \frac{1}{N + 1} \sum_{j = 1}^N | f_j(0, t z_1) |^2 = \eta(f_1(0, t z_1), \dots, f_N(0, t z_1)) \in \mathcal{H}^t
\end{align*}
Also notice that $\psi_0(0) = 1$, by construction. Similarly, the function defined in the neighborhood $B_{\frac{1}{C}}(0) \subset \mathbb{C}^N$ of $0$ by
\begin{align*}
    \delta(w) &= \sum_{j = 0}^n c'(s, j) (j + m)! \left(1 - C || w ||^2 \right)^{- j - m - 1} =\\
    &= \sum_{j = 0}^n c'(s, j) (j + m)! \sum_{l = 0}^\infty \binom{m + j + l}{l} C^l || w ||^{2l} =\\
    &= \sum_{l = 0}^\infty \left( C^l \sum_{j = 0}^n c'(s, j) (j + m)! \binom{m + j + l}{l} \right) \sum_{| \gamma | = l} \frac{l!}{\gamma !} w^\gamma \overline{w}^\gamma
\end{align*}
lies in $\mathcal{G}^N$, and its analytic continuation, given in a neighborhood $U \times \overline{U}$ of $(0, 0)$ by
\begin{align*}
    \widetilde{\delta}(w, \overline{v}) &= \sum_{j = 0}^n c'(s, j) (j + m)! \left(1 - C \langle w, v \rangle \right)^{- j - m - 1}
\end{align*}
is holomorphic. Given the polynomial $Q$ on $(\mathbb{C}^N \times \overline{\mathbb{C}^N}) \times \mathbb{C}$ 
\begin{align*}
    Q((X, \overline{Y}), \alpha) = \alpha (1 - C \langle X, Y \rangle)^{m + n + 1} - \sum_{j = 0}^n c'(s, j) (j + m)! (1 - \langle X, Y \rangle)^{n - j}
\end{align*}
then $Q((w, \overline{v}), \widetilde{\delta}(w, \overline{v})) = 0$ and $\delta \in \mathcal{N}^N$. We conclude
\begin{align*}
    \psi_1 := \sum_{j = 0}^n c'(s, j) (j + m)! \left(1 - C | t |^2 \right)^{- j - m - 1} = \delta(\text{id}_\mathbb{C}, 0, \dots, 0) \in \mathcal{H}^N
\end{align*}
and $\psi_1$ is not constant. With the introduced notation, \eqref{EqHartogsNashalg} reads $\psi_0 = \psi_1^{-\frac{\lambda}{N + 1}}$, so Lemma \ref{LemmaNash} applies, yielding $\lambda \in \mathbb{Q}$. This is the desired constriction. 

Consider the complex submanifold $(\mathbb{B}^n)_{m, s} \cap \{ z \in \mathbb{C}^{n + m} : z_1 = \dots = z_n = 0 \} = \{ 0 \} \times B_{\frac{1}{C}}(0)$ of $(\mathbb{B}^n)_{m, s}$, endowed with the real analytic K{\"a}hler metric $i^*g_{\mathbb{B}^n_{m, s}}$ induced by the inclusion $i \colon \{ 0 \} \times B_{\frac{1}{C}}(0) \to (\mathbb{B}^n)_{m, s}$. By \eqref{Eqdiastasispreserved}, the diastasis potential at $0$ for $i^*g_{\mathbb{B}^n_{m, s}}$ is given by $D_0 \circ i$, and
\begin{align*}
    f \circ i \colon (\{ 0 \} \times B_{\frac{1}{C}}(0), \lambda (i^*g_{\mathbb{B}^n_{m, s}})) \to (\mathbb{B}^N, g_{\mathbb{B}^N})
\end{align*}
is a K{\"a}hler immersion. Hence, in view of Theorem \ref{thmCalabicriterion}, we aim to find the expansion at $(0, 0)$ of the function
\begin{equation}\label{Eqdiastasissubmanif1}
    1 - e^{- \frac{D^\lambda_0(0, \xi)}{N + 1}} = 1 - \left( \sum_{j = 0}^n c'(s, j) (j + m)! \left(1 - C || \xi ||^2 \right)^{- j - m - 1} \right)^{-\frac{\lambda}{N + 1}} 
\end{equation}
Rewrite \eqref{Eqdiastasissubmanif1} as
\begin{equation}\label{Eqdiastasissubmanif2}
    \begin{split}
        1 - &e^{- \frac{D^\lambda_0(0, \xi)}{N + 1}} = 1 - (1 - C || \xi ||^2)^{\frac{\lambda(n + m + 1)}{N + 1}} \times\\
        &\hspace{5mm} \times \left( 1 - \left( 1 - \sum_{j = 0}^n c'(s, j) (j + m)! \left(1 - C || \xi ||^2 \right)^{n - j} \right) \right)^{-\frac{\lambda}{N + 1}}
    \end{split}
\end{equation}
Since $\left( 1 - \sum_{j = 0}^n c'(s, j) (j + m)! \left(1 - C || \xi ||^2 \right)^{n - j} \right) |_{\xi = 0} = 0$, by continuity there is a neighborhood $U \subset B_{\frac{1}{C}}(0)$ of $0 \in \mathbb{C}^m$ such that $\forall\ \xi \in U$
\begin{align*}
    \left| 1 - \sum_{j = 0}^n c'(s, j) (j + m)! \left(1 - C || \xi ||^2 \right)^{n - j} \right| < 1
\end{align*}
We can then apply the general binomial expansion to get
\begin{equation}\label{Eqdiastasissubmanif3a}
    \begin{split}
        &\left( 1 - \left( 1 - \sum_{j = 0}^n c'(s, j) (j + m)! \left(1 - C || \xi ||^2 \right)^{n - j} \right) \right)^{-\frac{\lambda}{N + 1}} =\\
        &\hspace{10mm} = \sum_{l = 0}^\infty \binom{\frac{\lambda}{N + 1} + l - 1}{l} \left( 1 - \sum_{j = 0}^n c'(s, j) (j + m)! \left(1 - C || \xi ||^2 \right)^{n - j} \right)^l
    \end{split}
\end{equation}
and the multinomial expansion to get
\begin{equation}\label{Eqdiastasissubmanif3b}
    \begin{split}
        &\left( 1 - \sum_{j = 0}^n c'(s, j) (j + m)! \left(1 - C || \xi ||^2 \right)^{n - j} \right)^l =\\
        &\hspace{10mm} = \sum_{| a | = l} \frac{l!}{a!} (-1)^{l - a_{n + 2}} \prod_{j = 1}^{n + 1} (c'(s, j - 1))^{a_j} (1 - C || \xi ||^2)^{(n - j)a_j} =\\
        &\hspace{10mm} = \sum_{| a | = l} \frac{l!}{a!} (-1)^{l - a_{n + 2}} (1 - C || \xi ||^2)^{\sum_{k = 1}^{n + 1}(n - k)a_k} \prod_{j = 1}^{n + 1} (c'(s, j - 1))^{a_j}
    \end{split}
\end{equation}
For any multi-index $a = (a_1, \dots, a_{n + 2})$, set $c(a) := \frac{(-1)^{| a | + 1 - a_{n + 2}}}{a!}\prod_{j = 1}^{n + 1} (c'(s, j - 1))^{a_j}$. Insert \eqref{Eqdiastasissubmanif3a} and \eqref{Eqdiastasissubmanif3b} into the right-hand side of \eqref{Eqdiastasissubmanif2} to get
\begin{equation}\label{Eqdiastasissubmanif4}
    \begin{split}
        1 - e^{- \frac{D^\lambda_0(0, \xi)}{N + 1}} &= \sum_{l = 1}^\infty \binom{\frac{\lambda}{N + 1} + l - 1}{l} \sum_{| a | = l} l!\ c(a) \times\\
        &\hspace{5mm} \times (1 - C || \xi ||^2)^{\frac{\lambda(n + m + 1)}{N + 1} + \sum_{k = 1}^{n + 1}(n - k)a_k} =\\
        &= \sum_{l = 1}^\infty \binom{\frac{\lambda}{N + 1} + l - 1}{l} \sum_{| a | = l} l!\ c(a) \times\\
        &\hspace{5mm} \times \sum_{v = 0}^\infty \binom{\frac{\lambda(n + m + 1)}{N + 1} + \sum_{k = 1}^{n + 1}(n - k)a_k}{v} (-1)^v C^v || \xi ||^{2v}
    \end{split}
\end{equation}
By the absolute convergence of the generalized binomial series, we can interchange the order of the two series in \eqref{Eqdiastasissubmanif4}. Thus, set $\alpha(v) := \sum_{l = 1}^\infty \binom{\frac{\lambda}{N + 1} + l - 1}{l} \sum_{| a | = l} l!\ c(a) \binom{\frac{\lambda(n + m + 1)}{N + 1} + \sum_{k = 1}^{n + 1}(n - k)a_k}{v} (-1)^v C^v$ to obtain
\begin{equation}\label{Eqdiastasissubmanif5}
    \begin{split}
        1 - e^{- \frac{D^\lambda_0(0, \xi)}{N + 1}} = \sum_{v = 0}^\infty \alpha(v)\ || \xi ||^{2v} = \sum_{v = 0}^\infty \alpha(v) \sum_{| \gamma | = v} \frac{v!}{\gamma !} \xi^\gamma \overline{\xi^\gamma}
    \end{split}
\end{equation}
If we now switch to the multi-index notation introduced in \ref{sec2}, \eqref{Eqdiastasissubmanif5} becomes 
\begin{equation}
    1 - e^{- \frac{D^\lambda_0(0, \xi)}{N + 1}} = \sum_{r = 0}^\infty \beta(r)\ \xi^{m_r} \overline{\xi^{m_r}}
\end{equation}
where $\beta(0) = 0$ and $\forall\ r \geq 1$
\begin{align*}
   \beta(r) = \sum_{l = 1}^\infty \binom{\frac{\lambda}{N + 1} + l - 1}{l} \sum_{| a | = l} l!\ c(a) \binom{\frac{\lambda(n + m + 1)}{N + 1} + \sum_{k = 1}^{n + 1}(n - k)a_k}{| m_r |} (-C)^{| m_r |} \frac{| m_r |!}{m_r !}
\end{align*}
Hence, the infinite-dimensional matrix $(s_{j k}(0))_{j, k}$ of the expansion of the function $(N + 1)\big(1 - e^{- \frac{D^\lambda_0(0, \xi)}{N + 1}}\big)$ is given by
\begin{align*}
    (s_{j k}(0))_{j, k} = \text{diag}(0, (N + 1) \beta(1), (N + 1) \beta(2), \dots)
\end{align*}
and by Theorem \ref{thmCalabicriterion}, it is positive semidefinite of rank at most $N$. In particular $\exists\ L \in \mathbb{N}^*$ such that $s_{rr}(0) = 0\ \forall\ r \geq L$, or in other words, $(N + 1)\big(1 - e^{- \frac{D^\lambda_0(0, \xi)}{N + 1}}\big)$ is a polynomial. Applying a change of variable, by \eqref{Eqdiastasissubmanif2} we infer that
\begin{equation}\label{EqHartogsfinal1}
    (1 - C X)^{\frac{\lambda(n + m + 1)}{N + 1}} \left( \sum_{j = 0}^n c'(s, j) (j + m)! \left(1 - C X \right)^{n - j} \right)^{-\frac{\lambda}{N + 1}}
\end{equation}
is also a polynomial. Now, being $\lambda \in \mathbb{Q}$ and $\lambda > 0$, there are $\delta, \varepsilon \in \mathbb{N}^*$ for which $\frac{\lambda}{N + 1} = \frac{\delta}{\varepsilon}$. If we set
\begin{align*}
    T_1(X) &:= (1 - C X)^{\delta(n + m + 1)}\ ,\\
    T_2(X) &:= \left( \sum_{j = 0}^n c'(s, j) (j + m)! \left(1 - C X \right)^{n - j} \right)^{\delta}
\end{align*}
then the rational function $\frac{T_1}{T_2}$ is actually a polynomial, by \eqref{EqHartogsfinal1}. However, $T_2$ does not divide $T_1$, otherwise, $T_1$ and $T_2$ would share a factor and the only irreducible factor of $T_1$ is $1 - C X$, which does not divide $T_2$ because \eqref{Eqtopconstant} gives
\begin{align*}
    T_2\left( \frac{1}{C} \right) = \left( c'(s, n) (n + m)! \right)^\delta \neq 0
\end{align*}
It follows that $T_2$ must be constant: $T_2 \equiv T_2(0) = (\sum_{j = 0}^n c'(s, j) (j + m)!)^\delta = 1$. Thus
\begin{align*}
    \sum_{j = 0}^{n - 1} c'(s, j) (j + m)! (1 - C X)^{n - j} + c'(s, n) (n + m)! - 1 \equiv 0
\end{align*}
which, by linear independence of $\{ 1, 1 - C X, \dots, (1 - C X)^n \}$, results into $c(s, j) = 0$ $\forall\ j = 0, \dots, n - 1$. Then, by definition
\begin{equation}\label{EqHartogsfinal2}
    b(k) = c(s, n) (k + 1)_n
\end{equation}
and due to \eqref{Eqtopconstant} \eqref{EqHartogsfinal2}, if $Z(b(k)), Z((k + 1)_n)$ denote the zero locus of $b(k), (k + 1)_n$
\begin{equation*}\label{EqHartogsfinal3}
    \left\{ - \frac{l}{(n + 1)s} : l = 1, \dots, n \right\} = Z(b(k)) = Z((k + 1)_n) = \{ - l : l = 1, \dots, n \} 
\end{equation*}
Since $Z((k + 1)_n) \subset \mathbb{Z}^-$, $s$ is forced to be positive, so $-\frac{1}{(n + 1)s} = \max Z(b(k)) = \max Z((k + 1)_n) = -1$, and finally $s = \frac{1}{n + 1}$, as desired.

\subsection{Proof of Theorem \ref{thm3}}

Let $f \colon (E(p, q, \Omega, k), \lambda g_{E}) \to (\mathbb{B}^N, g_{\mathbb{B}^N}),\ f = (f_1, \dots, f_N)$, be a K{\"a}hler immersion. As in the proof of Theorem \ref{thm2}, assume $0 \in \Omega$ (so $0 \in E(p, q, \Omega, k)$) and $f(0) = 0$. The inclusion $\iota \colon \Omega \to E(p, q, \Omega, k): \iota(z) = (z, 0, 0)$ is a holomorphic immersion and by \eqref{EqBergmanEgg}
\begin{align*}
    (\iota^*K_E)(z) = \frac{1}{p! q!} \Lambda^{(p - 1), (q - 1)} \left( 0, 0 \right) (\text{vol}(\Omega) K_\Omega(z))^{\frac{p}{g} + \frac{q}{kg} + 1}
\end{align*}
which implies $\iota^*\omega_E = (\frac{p}{g} + \frac{q}{kg} + 1) \omega_\Omega$, where $\omega_E, \omega_\Omega$ are the K{\"a}hler forms associated to the Bergman metrics. Hence, $\iota \colon (\Omega, \big(\frac{p}{g} + \frac{q}{kg} + 1\big) g_\Omega) \to (E(p, q, \Omega, k), g_E)$ is a K{\"a}hler immersion. It follows that
\begin{align*}
    f \circ \iota \colon \left(\Omega, \lambda \Big(\frac{p}{g} + \frac{q}{kg} + 1\Big) g_\Omega\right) \to (\mathbb{B}^N, g_{\mathbb{B}^N})\ \text{is a K{\"a}hler immersion}
\end{align*}
which implies that $\Omega \cong \mathbb{B}^n$, by Theorem $2$ of \cite{LoiScala2} (applied after rescaling). Now let $\varphi \colon \Omega \to \mathbb{B}^n$ be a biholomorphism. Applying \eqref{EqtransfBergman}, we see that $E(p, q, \Omega, k)$ is biholomorphic to
\begin{equation*}\label{EggOmegatoball1}
    \{ (z, \xi_1, \xi_2) \in \Omega \times \mathbb{C}^p \times \mathbb{C}^q : |\det(J(\varphi)(z))|^{\frac{2}{g}} (|| \xi_1 ||^2 + || \xi_2 ||^{2k}) < N_{\mathbb{B}^n}(\varphi(z)) \}
\end{equation*}
Denote this domain by $E'$. As in the proof of Theorem \ref{thm2}, by Lemma \ref{Lemmalogarithm} there is a holomorphic exponential form for $\det(J(\varphi))$ and we can define a biholomorphism
\begin{align*}
    \Psi \colon E' \to E(p, q, \mathbb{B}^n, k) : (z, \xi_1, \xi_2) \mapsto \Big( \varphi(z), \det(J(\varphi)(z))^\frac{1}{g} \xi_1, \det(J(\varphi)(z))^\frac{1}{kg} \xi_2 \Big)
\end{align*}
so we reduced to the case of a K{\"a}hler immersion $(E(p, q, \mathbb{B}^n, k), \lambda g_E) \to (\mathbb{B}^N, g_{\mathbb{B}^N})$. However, $(E(p, q, \mathbb{B}^n, k), \lambda g_E) \cong (\mathbb{B}^{n + p})_{q, \frac{1}{k}}$, and Theorem \ref{thm2} applies.

\bibliography{Bibliography}

\end{document}